\documentclass[11pt,leqno]{article}

\usepackage[active]{srcltx}

\usepackage{amsfonts,latexsym,amsmath,amssymb,amsthm}
\usepackage{fullpage}
\newtheorem{theorem}{Theorem}[section]
\newtheorem{lemma}[theorem]{Lemma}

\newtheorem{proposition}[theorem]{Proposition}
\newtheorem{corollary}[theorem]{Corollary}
\newtheorem{remark}[theorem]{Remark}

\theoremstyle{definition}
\newtheorem{definition}[theorem]{Definition}

\theoremstyle{remark}

\newtheorem*{note*}{Note}

\numberwithin{equation}{section}

\makeatletter
\newcommand{\rank}{\mathop{\operator@font rank}}
\newcommand{\conv}{\mathop{\operator@font conv}}
\newcommand{\vol}{\mathop{\operator@font vol}}
\newcommand{\onetagright}{\tagsleft@false}
\makeatother

\newcommand{\ls}{\leqslant}
\newcommand{\gr}{\geqslant}

\usepackage{color}

\usepackage{hyperref}

\begin{document}
\small

\title{\bf Half-space depth of log-concave probability measures}

\medskip

\author{Silouanos Brazitikos, Apostolos Giannopoulos and Minas Pafis}

\date{\small \textit{Dedicated to the memory of Dimitris Gatzouras}}

\maketitle

\begin{abstract}\footnotesize Given a probability measure $\mu $ on ${\mathbb R}^n$, Tukey's half-space depth is defined
for any $x\in {\mathbb R}^n$ by $\varphi_{\mu }(x)=\inf\{\mu (H):H\in {\cal H}(x)\}$, where ${\cal H}(x)$ is the set
of all half-spaces $H$ of ${\mathbb R}^n$ containing $x$. We show that if $\mu $ is a non-degenerate log-concave
probability measure on ${\mathbb R}^n$ then
$$e^{-c_1n}\ls \int_{\mathbb{R}^n}\varphi_{\mu }(x)\,d\mu(x) \ls e^{-c_2n/L_{\mu}^2}$$
where $L_{\mu }$ is the isotropic constant of $\mu $ and $c_1,c_2>0$ are absolute constants. The proofs
combine large deviations techniques with a number of facts from the theory of $L_q$-centroid bodies of
log-concave probability measures. The same ideas lead to general estimates for the expected measure of random
polytopes whose vertices have a log-concave distribution.
\end{abstract}

\section{Introduction}

Let $\mu $ be a probability measure on ${\mathbb R}^n$. For any $x\in {\mathbb R}^n$ we denote by ${\cal H}(x)$ the set
of all half-spaces $H$ of ${\mathbb R}^n$ containing $x$. The function
$$\varphi_{\mu }(x)=\inf\{\mu (H):H\in {\cal H}(x)\}$$
is called Tukey's half-space depth. The first work in statistics where some form of the half-space depth appears
is an article of Hodges \cite{Hodges-1955} from 1955. Tukey introduced the half-space depth for data sets in \cite{Tukey-1975} as a
tool that enables efficient visualization of random samples in the plane. The term ``depth" also comes
from Tukey's article.  A formal definition of the half-space depth as a way to distinguish points that fit the overall
pattern of a multivariable probability distribution and to obtain an efficient description, visualization, and
nonparametric statistical inference for multivariable data, was given by
Donoho and Gasko in \cite{Donoho-Gasko-1982} (see also \cite{Small-1987}).
We refer the reader to the survey article of Nagy, Sch\"{u}tt and Werner
\cite{Nagy-Schutt-Werner-2019} for an overview of this topic, with an emphasis on its connections
with convex geometry, and many references.

In the first part of this article we study the expectation of the half-space depth in the context of log-concave probability measures.
In what follows, these are the Borel probability measures $\mu$ on $\mathbb R^n$ that satisfy $\mu(\lambda
A+(1-\lambda)B) \gr \mu(A)^{\lambda}\mu(B)^{1-\lambda}$ for any compact subsets $A,B\subseteq {\mathbb R}^n$ and any
$\lambda \in (0,1)$, as well as the non-degeneracy condition $\mu (H)<1$ for every hyperplane $H$ in ${\mathbb R}^n$.
The question whether there exists an absolute constant $c\in (0,1)$ such that
\begin{equation}\label{question-1}{\mathbb E}_{\mu }(\varphi_{\mu }):=\int_{{\mathbb R}^n}\varphi_{\mu }(x)\,d\mu (x)\ls c^n\end{equation}
for all $n\gr 1$ and all log-concave probability measures $\mu $ on ${\mathbb R}^n$ was asked in \cite{mathoverflow}
in connection with stochastic separability and applications to machine learning  and
error-correction mechanisms in artificial intelligence systems; for the origin of the
question we refer to \cite{GGT-2023} and to the references therein. In the context of asymptotic geometric analysis,
the validity of \eqref{question-1} implies that if $m\ls C^n$, where $C>1$ is an absolute constant,
then a set of $m$ independent random points with a log-concave distribution has, with probability close to $1$, the property that every point in the set can be separated from all others by a hyperplane.

Our first result shows that \eqref{question-1} holds true modulo the isotropic constant $L_{\mu}$ of $\mu$,
defined in \eqref{eq:L}.

\begin{theorem}\label{th:intro-1}
Let $\mu$ be a log-concave probability measure on ${\mathbb R}^n$, $n\gr n_0$. Then,
${\mathbb E}_{\mu }(\varphi_{\mu }) \ls \exp\left(-cn/L_{\mu}^2\right)$
where $L_{\mu }$ is the isotropic constant of $\mu $ and $c>0$, $n_0\in {\mathbb N}$ are absolute constants.
\end{theorem}

Background information on isotropic log-concave probability measures and the isotropic constant is
provided in Section~\ref{section-2}. The well-known hyperplane conjecture asks whether there exists an absolute constant $C>0$ such that $L_n\ls C$ for every $n\gr 2$, where
$$L_n=\sup\{L_{\mu }:\mu \ \hbox{is an isotropic log-concave probability measure on}\ {\mathbb R}^n\}.$$
The best known upper bound, due to Klartag \cite{Klartag-2023}, asserts that $L_n\ls C\sqrt{\ln n}$ for
some absolute constant $C>0$, therefore Theorem~\ref{th:intro-1} shows that
$${\mathbb E}_{\mu }(\varphi_{\mu }) \ls \exp\left(-cn/\ln n\right)$$
provided that $n$ is large enough. The quantity ${\mathbb E}_{\mu }(\varphi_{\mu })$ is affinely
invariant and hence for the proof of Theorem~\ref{th:intro-1} we may assume that $\mu $ is isotropic.
Actually, we obtain Theorem~\ref{th:intro-1} as a special case of a
more general result which is presented in Section~\ref{section-3}.

\begin{theorem}\label{th:intro-2}
Let $\mu$ and $\nu $ be two isotropic log-concave probability measures on ${\mathbb R}^n$, $n\gr n_0$. Then,
$${\mathbb E}_{\nu }(\varphi_{\mu }):=\int_{\mathbb{R}^n}\varphi_{\mu }(x)\,d\nu(x)\ls \exp\left(-cn/L_{\nu }^2\right),$$
where $c>0$, $n_0\in {\mathbb N}$ are absolute constants.
\end{theorem}

The proof of Theorem~\ref{th:intro-2} starts with the known estimate $\varphi_{\mu}(x)\ls\exp(-\Lambda_{\mu}^{\ast}(x))$
where $\Lambda_{\mu}^{\ast}$ is the Cram\'{e}r transform of $\mu$ (defined in Section~\ref{section-2}),
and actually establishes the stronger inequality
\begin{equation}\label{eq:stronger}\int_{{\mathbb R}^n}e^{-\Lambda_{\mu}^{\ast}(x)}d\nu(x)
\ls \exp\left(-cn/L_{\nu}^2\right),\end{equation}
exploiting upper bounds for the volume of the sets $B_t(\mu)=\{x\in {\mathbb R}^n:\Lambda_{\mu}^{\ast}(x)\ls t\}$.
The assumption that both $\mu $ and $\nu $ are isotropic is not necessary. One can consider
a different type of normalization. We discuss this matter in Section~\ref{section-2} and we state another version of
Theorem~\ref{th:intro-2} that might be useful (see Theorem~\ref{th:variant}). In any case, setting $\nu =\mu $ we obtain Theorem~\ref{th:intro-1} as an immediate consequence of any of these statements.

In Section~\ref{section-4} we show that, apart from the value of the isotropic constant $L_{\mu }$,
the exponential estimate provided by Theorem~\ref{th:intro-1} is sharp.

\begin{theorem}\label{th:intro-3}Let $\mu $ be a log-concave probability measure on ${\mathbb R}^n$. Then,
$$\int_{{\mathbb R}^n}\varphi_{\mu}(x)d\mu(x)\gr e^{-cn},$$
where $c>0$ is an absolute constant.
\end{theorem}

The proof of Theorem~\ref{th:intro-3} makes use of several facts about isotropic
log-concave probability measures. In the case where $\mu$ is the uniform measure on a convex body
$K$ of volume $1$ in ${\mathbb R}^n$, one can show that $\varphi_{\mu}(x)\gr e^{-c_1n}$ for all
$x\in\tfrac{1}{2}K$ and then simply apply Markov's inequality and use the fact that $\left|\tfrac{1}{2}K\right|=2^{-n}$.
When $\mu$ is an arbitrary log-concave probability measure on ${\mathbb R}^n$, in order to obtain
the same exponential in the dimension lower bound we have to exploit the family of the one-sided
$L_t$-centroid bodies of $\mu$. More precisely, we use the fact that in order to have the lower bound
$\varphi_{\mu}(x)\gr e^{-c_1n}$ we may use, instead of $\tfrac{1}{2}K$, the convex body $\tfrac{1}{2}Z_t^+(\mu)$ with e.g. $t=5n$, where $Z_t^+(\mu)$ is the one-sided $L_t$-centroid body of $\mu$, and we establish an appropriate lower bound for $\mu\left(\tfrac{1}{2}Z_{5n}^+(\mu)\right)$. This last estimate requires the use of some other families of convex sets that are
associated with a log-concave probability measure; these are introduced in the next section as well as in
Section~\ref{section-4}. For the reader's convenience we present first the proof of Theorem~\ref{th:intro-3} in the simpler
case where $\mu $ is the uniform measure on a convex body $K$ in ${\mathbb R}^n$ and then in the general case of an
arbitrary log-concave probability measure.

\smallskip

In the second part of this article we consider the question to obtain uniform upper and lower thresholds
for the expected measure of a random polytope defined as the convex hull of independent random points with a log-concave distribution. The general formulation of the problem is the following. Given a log-concave probability measure $\mu $ on ${\mathbb R}^n$
we consider independent random points $X_1,X_2,\ldots $
in ${\mathbb R}^n$ distributed according to $\mu $ and for any $N>n$ we consider the random polytope
$$K_N={\rm conv}\{X_1,\ldots ,X_N\}$$
and the expectation ${\mathbb E}_{\mu^N}[\mu (K_N)]$. Tukey's half-space depth plays a crucial role in the study of these
random polytopes and of their threshold behavior, starting with the classical work of Dyer, F\"{u}redi and McDiarmid who established in \cite{DFM} a sharp threshold for the expected volume of random polytopes with vertices uniformly distributed in the discrete cube $E_2^n=\{-1,1\}^n$ or in the solid cube $B_{\infty }^n=[-1,1]^n$. They proved that in the first case, if $\kappa =\ln 2-\tfrac{1}{2}$ then for every $\varepsilon \in (0,\kappa )$ one has the upper threshold
\begin{equation}\label{eq:up-k}\lim_{n\rightarrow\infty }\sup\left\{2^{-n} {\mathbb E}|K_N|\colon N\ls \exp((\kappa -\varepsilon
)n)\right\}=0\end{equation} and the lower threshold \begin{equation}\label{eq:low-k}\lim_{n\rightarrow\infty
}\inf\left \{ 2^{-n} {\mathbb E}|K_N|\colon N\gr \exp((\kappa +\varepsilon
)n)\right\}=1.\end{equation}
A similar result holds true for the expected volume of random polytopes with vertices uniformly distributed in
the cube $B_{\infty }^n$; the corresponding value of the
constant $\kappa $ is $\kappa =\ln(2\pi)-\gamma -\tfrac{1}{2}$, where $\gamma $ is Euler's constant. Half-space depth
plays a key role in the proof of these results: the starting point for the proof of the upper and lower threshold
are variants of Lemma~\ref{lem:upper-1} and Lemma~\ref{lem:inclusion} respectively. Further sharp thresholds
(meaning that there exists some constant $\kappa =\kappa_{\mu}$ such that the expected volume of $K_N$ changes behavior
around $N=\exp(\kappa_{\mu}n)$) have been given in a number of other special cases; see \cite{Gatzouras-Giannopoulos-2009}
for the case where $X_i$ have independent identically distributed coordinates supported on a bounded interval, and the articles \cite{Pivovarov-2007} and \cite{Bonnet-Chasapis-Grote-Temesvari-Turchi-2019}, \cite{Bonnet-Kabluchko-Turchi-2021} for a number of cases where $X_i$ have rotationally invariant densities. All these works follow the same strategy and use estimates for the half-space depth.
Non-sharp, both of them exponential in the dimension, upper and lower thresholds are obtained in \cite{Frieze-Pegden-Tkocz-2020} for the case where $X_i$ are uniformly distributed in a simplex. All these results suggest that, at least in the case where $\mu=\mu_K$ is the uniform measure on a high-dimensional convex body, the expectation ${\mathbb E}_{\mu^N}[\mu (K_N)]$
of the measure of $K_N$ exhibits a threshold with constant $\kappa_{\mu} =\frac{1}{n}{\mathbb E}_{\mu }(\Lambda_{\mu}^{\ast })$,
where $\Lambda_{\mu}^{\ast}$ is the Cram\'{e}r transform of $\mu$, in the sense that the following statement
might be true: given $\delta\in \left(0,\tfrac{1}{2}\right)$, there exists
$n_0(\delta,\varepsilon )\in {\mathbb N}$ such that if $n\gr n_0$ and $K$ is a convex body in ${\mathbb R}^n$ then
$$\sup\left\{{\mathbb E}_{\mu^N}[\mu (K_N)]\colon N\ls \exp((\kappa_{\mu}-\varepsilon) n)\right\}\ls\delta $$
and
$$\inf\left \{{\mathbb E}_{\mu^N}[\mu (K_N)]\colon N\gr \exp((\kappa_{\mu}+\varepsilon) n)\right\}\gr 1-\delta $$
for some $\varepsilon =c(n,\delta)\kappa_{\mu}$ with $\lim\limits_{n\to\infty}c(n,\delta)=0$. Some steps in this direction have been made in \cite{BGP-threshold}. Note that by \eqref{eq:stronger} and Jensen's inequality one has that $\kappa_{\mu}\gr c/L_n^2$ for every log-concave probability measure $\mu $ on ${\mathbb R}^n$.

Here, we are interested in uniform upper and lower thresholds for the class of all log-concave probability
measures. The question that we study is to find a constant $N_1(n)$, depending only on $n$ and as large as possible, so that
$$\sup_{\mu }\Big(\sup\Big\{{\mathbb E}_{\mu^N}[\mu (K_N)]:N\ls N_1(n)\Big\}\Big)\longrightarrow 0$$ as $n\to\infty $
and a second constant $N_2(n)$, depending only on $n$ and as small as possible,
so that
$$\inf_{\mu}\Big(\inf\Big\{{\mathbb E}_{\mu^N}[\mu (K_N)]:N\gr N_2(n)\Big\}\Big)\longrightarrow 0$$ as $n\to\infty $,
where the supremum and the infimum are over all log-concave probability measures. We shall call the first type of result a ``uniform upper threshold" and the second type a ``uniform lower threshold".

Such uniform upper and lower thresholds were obtained recently by Chakraborti, Tkocz and Vritsiou in \cite{Chakraborti-Tkocz-Vritsiou-2021} for some families of distributions. They showed that if $\mu $ is an even log-concave
probability measure supported on a convex body $K$ in ${\mathbb R}^n$ and if $X_1,X_2,\ldots $ are independent random points distributed according to $\mu $, then for any $n<N\ls \exp (c_1n/L_{\mu }^2)$ we have that
\begin{equation}\label{eq:tk-1}\frac{{\mathbb E}_{\mu^N}(|K_N|)}{|K|} \ls \exp\left(-c_2n/L_{\mu }^2\right),\end{equation}
where $c_1,c_2>0$ are absolute constants. We obtain an upper threshold for a pair of log-concave measures $\mu $ and $\nu $,
if they can be simultaneously put in the isotropic position.

\begin{theorem}\label{th:intro-4}
Let $\mu$ and $\nu $ be isotropic log-concave probability measures on ${\mathbb R}^n$. Let $X_1,X_2,\ldots $ be independent
random points in ${\mathbb R}^n$ distributed according to $\mu $ and for any $N>n$ consider the random polytope
$K_N={\rm conv}\{X_1,\ldots ,X_N\}$. Then, for any $N\ls \exp (c_1n/L_{\nu }^2)$ we have that
$${\mathbb E}_{\mu^N}(\nu (K_N)) \ls 2\exp\left(-c_2n/L_{\nu }^2\right),$$
where $c_1,c_2>0$ are absolute constants.
\end{theorem}

As a corollary of Theorem~\ref{th:intro-4} we get:

\begin{corollary}\label{cor_intro-1}There exists an absolute constant $c>0$ such that if $N_1(n)=\exp (cn/L_n^2)$ then
$$\sup_{\mu }\Big(\sup\Big\{{\mathbb E}_{\mu^N}[\mu (K_N)]:N\ls N_1(n)\Big\}\Big)\longrightarrow 0$$
as $n\to\infty $, where the first supremum is over all log-concave probability measures $\mu $ on ${\mathbb R}^n$.
\end{corollary}

The proof of Theorem~\ref{th:intro-4} exploits some of the ideas that are used for the proof of \eqref{eq:tk-1}
in \cite{Chakraborti-Tkocz-Vritsiou-2021}: Lemma~\ref{lem:upper-1} is a variant of
a known idea which is often used for upper thresholds and is based again on the inequality
$\varphi_{\mu}(x)\ls\exp(-\Lambda_{\mu}^{\ast}(x))$. Then, one has to use upper bounds
for the volume of the sets $B_t(\mu)$. The assumption that both $\mu $ and $\nu $ are isotropic may be replaced
by different types of normalization. We discuss other versions of Theorem~\ref{th:intro-4} in Section~\ref{section-5}
and we show that one can recover \eqref{eq:tk-1} from these.

The uniform lower threshold which is established in \cite{Chakraborti-Tkocz-Vritsiou-2021} concerns the case
where $\mu $ is an even $\kappa $-concave measure on ${\mathbb R}^n$ with $0<\kappa <1/n$, supported on a convex body $K$ in ${\mathbb R}^n$. If $X_1,X_2,\ldots $ are independent random points in ${\mathbb R}^n$ distributed according to $\mu $ and $K_N={\rm conv}\{X_1,\ldots ,X_N\}$ as before, then for any $M\gr C$ and any
$N\gr \exp\left(\frac{1}{\kappa }(\ln n+2\ln M)\right)$ we have that
\begin{equation}\label{eq:tk-2}\frac{{\mathbb E}_{\mu^N}(|K_N|)}{|K|}\gr 1-\frac{1}{M},\end{equation}
where $C>0$ is an absolute constant.

Since the family of log-concave probability measures corresponds to the case $\kappa =0$, it is natural to ask for analogues of this result for $0$-concave, i.e. log-concave, probability measures. We obtain a
uniform lower threshold for the class of log-concave probability measures.

\begin{theorem}\label{th:intro-5}Let $\delta\in (0,1)$. Then,
$$\inf_{\mu }\Big(\inf\Big\{ {\mathbb E}_{\mu^N}\big[\mu ((1+\delta )K_N)\big]: N\gr \exp \big (C\delta^{-1}\ln \left(2/\delta \right)n\ln n\big )\Big\}\Big)\longrightarrow 1$$ as $n\to\infty $, where the first infimum is over all log-concave probability measures $\mu $ on ${\mathbb R}^n$ with barycenter at the origin, and $C>0$ is an absolute constant.
\end{theorem}

The proof of Theorem~\ref{th:intro-5} exploits the half-space depth as follows. By a known fact, Lemma~\ref{lem:inclusion},
roughly speaking it suffices to have a good lower bound for $\varphi_{\mu}(x)$ on a set $A\subset {\mathbb R}^n$ of measure close to $1$.
We show that if $\mu$ has its barycenter at the origin then, as in the proof of Theorem~\ref{th:intro-3}, the role of $A$ can be played by $(1+\delta)Z_t^+(\mu)$ where, this time, $t\gr C_{\delta }n\ln n$ and $C_{\delta }=C\delta^{-1}\ln\left(2/\delta\right)$. Theorem~\ref{th:intro-5} provides a weak threshold in the sense that we estimate the expectation ${\mathbb E}_{\mu^N}\big(\mu (1+\delta )K_N)$ (for an arbitrarily small but positive value of $\delta )$ while we would like to have a similar result for ${\mathbb E}_{\mu^N}\big[\mu (K_N)]$. One can ``remove the $\delta $-term",
however the dependence on $n$ becomes worse. More precisely, we show in Theorem~\ref{th:non-sharp} that there exists an absolute
constant $C>0$ such that
$$\inf_{\mu }\Big(\inf\Big\{ {\mathbb E}_{\mu^N}\big[\mu (K_N)\big]: N\gr \exp (C(n\ln n)^2u(n))\Big\}\Big)\longrightarrow 1$$
as $n\to\infty $, where the first infimum is over all log-concave probability measures $\mu $ on ${\mathbb R}^n$
and $u(n)$ is any function with $u(n)\to\infty $ as $n\to\infty $.

It should be noted that an exponential in the dimension lower threshold is not possible in full generality. For example, in the case
where $X_i$ are uniformly distributed in the Euclidean ball the sharp threshold for the problem is
$$\exp \left((1\pm \epsilon )\tfrac{1}{2}n\ln n\right),\qquad \epsilon >0.$$
See \cite{DFM0} where a related estimate first appears, and \cite{Pivovarov-2007},\cite{Bonnet-Chasapis-Grote-Temesvari-Turchi-2019} for sharp estimates; one more proof is given in \cite{BGP-threshold}.

\section{Notation and background information}\label{section-2}

In this section we introduce notation and terminology that we use throughout this work, and provide background
information on isotropic convex bodies and log-concave probability measures. We write $\langle\cdot ,\cdot\rangle $
for the standard inner product in ${\mathbb R}^n$ and denote the Euclidean norm by $|\cdot |$. In what follows, $B_2^n$ is the Euclidean unit ball, $S^{n-1}$ is the unit sphere, and $\sigma $ is the rotationally invariant probability measure on $S^{n-1}$. Lebesgue measure in ${\mathbb R}^n$ is denoted by $|\cdot |$. The letters $c, c^{\prime },c_j,c_j^{\prime }$ etc. denote absolute positive constants whose value may change from line to line. Whenever we write $a\approx b$, we mean that there exist absolute constants $c_1,c_2>0$ such that $c_1a\ls b\ls c_2a$. Similarly, if $A, B$ are sets, then $A \approx B$ will state that $c_1A\subseteq B \subseteq c_2 A$ for some absolute constants $c_1,c_2>0$. We refer to Schneider's book \cite{Schneider-book} for basic facts from the Brunn-Minkowski theory and to the book
\cite{AGA-book} for basic facts from asymptotic convex geometry. We also refer to \cite{BGVV-book} for more information on isotropic
convex bodies and log-concave probability measures.

\smallskip

\noindent {\bf 2.1. Convex bodies.} A convex body in ${\mathbb R}^n$ is a compact convex set $K\subset {\mathbb R}^n$ with non-empty interior.
In this work we often consider bounded convex sets $K$ in ${\mathbb R}^n$ with $0\in {\rm int}(K)$; since their closure is a convex body,
we shall call these sets convex bodies too. We say that $K$ is centrally symmetric if $-K=K$ and that $K$ is centered if the barycenter ${\rm bar}(K)=\frac{1}{|K|}\int_Kx\,dx$ of $K$ is at the origin.
We shall use the fact that if $K$ is a centered convex body in ${\mathbb R}^n$ then
\begin{equation}\label{eq:frad-1}\max_{y\in {\mathbb R}^n}|K\cap (y+\xi^{\perp })|_{n-1}\ls e\,|K\cap\xi^{\perp }|_{n-1}\end{equation}
for all $\xi\in S^{n-1}$, where $\xi^{\perp }=\{x\in {\mathbb R}^n:\langle x,\xi\rangle =0\}$ and $|\cdot|_{n-1}$
denotes $(n-1)$-dimensional volume. This is a result of Fradelizi; for a proof see \cite[Proposition~6.1.9]{BGVV-book}.
The radial function $\varrho_K$ of $K$ is defined for all $x\neq 0$ by $\varrho_K(x)=\sup \{\lambda >0:\lambda x\in K\}$ and
the support function of $K$ is given by $h_K(x) = \sup\{\langle x,y\rangle :y\in K\}$ for all $x\in {\mathbb R}^n$.
The polar body $K^{\circ }$ of a convex body $K$ in ${\mathbb R}^n$ with $0\in {\rm int}(K)$ is the convex body
\begin{equation*}
K^{\circ}:=\bigl\{y\in {\mathbb R}^n: \langle x,y\rangle \ls 1\;\hbox{for all}\; x\in K\bigr\}.
\end{equation*}
A convex body $K$ in ${\mathbb R}^n$ is called isotropic if it has volume $1$, it is centered, and its inertia matrix is a multiple of the
identity matrix: there exists a constant $L_K>0$, the isotropic constant of $K$, such that
\begin{equation*}\|\langle \cdot ,\xi\rangle\|_{L_2(K)}^2:=\int_K\langle x,\xi\rangle^2dx =L_K^2\end{equation*}
for all $\xi\in S^{n-1}$.

\smallskip

\noindent {\bf 2.2. Log-concave probability measures.} In this article, a Borel measure $\mu$ on $\mathbb R^n$ is called log-concave if
$\mu(H)<1$ for every hyperplane $H$ in ${\mathbb R}^n$ and $\mu(\lambda A+(1-\lambda)B) \gr \mu(A)^{\lambda}\mu(B)^{1-\lambda}$ for any compact subsets $A,B$ of ${\mathbb R}^n$ and any $\lambda \in (0,1)$. A theorem of Borell \cite{Borell-1974} shows that under these assumptions, $\mu $ has a log-concave density $f_{{\mu }}$. A function
$f:\mathbb R^n \rightarrow [0,\infty)$ is called log-concave if its support $\{f>0\}$ is a convex set in ${\mathbb R}^n$ and the restriction of $\ln{f}$ to it is concave. If $f$ has finite positive integral then there exist constants $A,B>0$ such that $f(x)\ls Ae^{-B|x|}$ for all $x\in {\mathbb R}^n$ (see \cite[Lemma~2.2.1]{BGVV-book}). In particular, $f$ has finite moments
of all orders. We say that $\mu $ is even if $\mu (-B)=\mu (B)$ for every Borel subset $B$ of ${\mathbb R}^n$ and that $\mu $ is centered if
\begin{equation*}
\int_{\mathbb R^n} \langle x, \xi \rangle d\mu(x) = \int_{\mathbb R^n} \langle x, \xi \rangle f_{\mu}(x) dx = 0
\end{equation*} for all $\xi\in S^{n-1}$. We shall use the fact that if $\mu $ is a centered log-concave
probability measure on ${\mathbb R}^n$ then
\begin{equation}\label{eq:frad-2}\|f_{\mu }\|_{\infty }\ls e^nf_{\mu }(0).\end{equation}
This is a result of Fradelizi; for a proof see \cite[Theorem~2.2.2]{BGVV-book}.
Note that if $K$ is a convex body in $\mathbb R^n$ then the Brunn-Minkowski inequality implies that the indicator function
$\mathbf{1}_{K} $ of $K$ is the density of a log-concave measure, the Lebesgue measure on $K$.

If $\mu $ is a log-concave measure on ${\mathbb R}^n$ with density $f_{\mu}$, we define the isotropic constant of $\mu $ by
\begin{equation}\label{eq:L}
L_{\mu }:=\left (\frac{\sup_{x\in {\mathbb R}^n} f_{\mu} (x)}{\int_{{\mathbb
R}^n}f_{\mu}(x)dx}\right )^{\frac{1}{n}} [\det \textrm{Cov}(\mu)]^{\frac{1}{2n}},\end{equation}
where $\textrm{Cov}(\mu)$ is the covariance matrix of $\mu$ with entries
\begin{equation*}\textrm{Cov}(\mu )_{ij}:=\frac{\int_{{\mathbb R}^n}x_ix_j f_{\mu}
(x)\,dx}{\int_{{\mathbb R}^n} f_{\mu} (x)\,dx}-\frac{\int_{{\mathbb
R}^n}x_i f_{\mu} (x)\,dx}{\int_{{\mathbb R}^n} f_{\mu}
(x)\,dx}\frac{\int_{{\mathbb R}^n}x_j f_{\mu}
(x)\,dx}{\int_{{\mathbb R}^n} f_{\mu} (x)\,dx}.\end{equation*} We say
that a log-concave probability measure $\mu $ on ${\mathbb R}^n$
is isotropic if it is centered and $\textrm{Cov}(\mu )=I_n$,
where $I_n$ is the identity $n\times n$ matrix. In this case, $L_{\mu }=\|f_{\mu }\|_{\infty }^{1/n}$.
For every $\mu $ there exists an affine transformation $T$
such that $T_{\ast }\mu $ is isotropic, where $T_{\ast}\mu $ is the push-forward of $\mu $ defined by $T_{\ast}\mu (A)=\mu (T^{-1}(A))$.
Note that a convex body $K$ of volume $1$ is isotropic if and only if the log-concave probability measure with density $L_K^n{\mathbf 1}_{K / L_K}$ is isotropic. The hyperplane conjecture asks if there exists an absolute constant $C>0$ such that
\begin{equation*}L_n:= \max\{ L_{\mu }:\mu\ \hbox{is an isotropic log-concave probability measure on}\ {\mathbb R}^n\}\ls C\end{equation*}
for all $n\gr 1$. Bourgain \cite{Bourgain-1991}
established  the upper bound $L_n\ls c\sqrt[4]{n}\ln n$; later, Klartag, in \cite{Klartag-2006},
improved this estimate to $L_n\ls c\sqrt[4]{n}$. In a breakthrough work, Chen \cite{C} proved that for any $\epsilon >0$
there exists $n_0(\epsilon )\in {\mathbb N}$ such that $L_n\ls n^{\epsilon}$ for every $n\gr n_0(\epsilon )$. Subsequently,
Klartag and Lehec \cite{Klartag-Lehec-2022} showed that $L_n\ls c(\ln n)^4$, and very recently Klartag \cite{Klartag-2023}
achieved the best known bound $L_n\ls c\sqrt{\ln n}$.

\smallskip

\noindent {\bf 2.3. Centroid bodies.} Let $\mu$ be a log-concave probability measure on $\mathbb R^n$. For any $t\gr 1$ we define the
$L_t$-centroid body $Z_t(\mu)$ of $\mu $ as the centrally symmetric convex body
whose support function is
\begin{equation*} h_{Z_t(\mu)}(y):=\left(
\int_{\mathbb R^n} |\langle x,y\rangle|^t f_{\mu}(x)dx \right)^{1/t},\qquad y\in {\mathbb R}^n.
\end{equation*} Note that $Z_t(\mu )$ is always centrally symmetric,
and $Z_t(T_{\ast}\mu)=T(Z_t(\mu ))$ for every $T\in GL(n)$ and $t\gr 1$.
Note also that a centered log-concave probability measure $\mu $ is isotropic if and only if $Z_2(\mu )=B_2^n$.
The next result of Paouris (see \cite[Theorem~5.1.17]{BGVV-book}) provides upper bounds for the volume
of the $L_t$-centroid bodies of isotropic log-concave probability measures.

\begin{theorem}\label{th:fact-1} If $\mu$ is a centered log-concave probability measure on $\mathbb R^n$, then for every
$2\ls t\ls n$ we have that
\begin{equation*}|Z_t(\mu)|^{1/n}\ls c\sqrt{t/n}[\det {\rm Cov}(\mu)]^{\frac{1}{2n}},\end{equation*}
where $c>0$ is an absolute constant. In particular, if $\mu $ is isotropic then $|Z_t(\mu)|^{1/n}\ls c\sqrt{t/n}$
for all $2\ls t\ls n$.
\end{theorem}

A variant of the $L_t$-centroid bodies of $\mu $ is defined as follows. For every $t\gr 1$ we consider the convex body
$Z_t^+(\mu )$ with support function
\begin{equation*}h_{Z_t^+(\mu )}(y)=\left (\int_{{\mathbb R}^n}\langle
x,y\rangle_+^tf_{\mu }(x)dx\right )^{1/t},\qquad y\in {\mathbb R}^n,\end{equation*} where $a_+=\max
\{a,0\}$. When $f_{\mu }$ is even, we have that $Z_t^+(\mu )=2^{-1/t}Z_t(\mu )$. In any case, we easily verify that
\begin{equation*}Z_t^+(\mu )\subseteq Z_t(\mu ).\end{equation*}
Moreover, if $\mu $ is isotropic then $Z_2^+(\mu )\supseteq cB_2^n$ for an absolute constant $c>0$.
One can also check that if $1\ls t<s$ then
\begin{equation*}\left(\frac{4}{e}\right)^{\frac{1}{t}-\frac{1}{s}}Z_t^+(\mu )\subseteq Z_s^+(\mu )
\subseteq c_1\left(\frac{4(e-1)}{e}\right)^{\frac{1}{t}-\frac{1}{s}}\frac{s}{t}Z_t^+(\mu ).\end{equation*}
The right-hand side inequality gives
\begin{equation}\label{eq:pz-1}\int_{{\mathbb R}^n}\langle x,\xi\rangle_+^{2t}f_{\mu }(x)dx =[h_{Z_{2t}^+(\mu )}(\xi )]^{2t}
\ls C^{2t}[h_{Z_{t}^+(\mu )}(\xi )]^{2t}=C^{2t}\left(\int_{{\mathbb R}^n}\langle x,\xi\rangle_+^tf_{\mu }(x)dx\right)^2,\end{equation}
for all $\xi\in S^{n-1}$, where $C>1$ is an absolute constant. For a proof of all these claims see \cite{Guedon-EMilman-2011}, where
the family of bodies $\tilde{Z}_t^+(\mu)=2^{1/t}Z_t^+(\mu)$ is considered. We have made the necessary adjustments in the inclusions
that we use.

\smallskip

\noindent {\bf 2.4. The bodies $B_t(\mu )$.} Let $\mu $ be a probability measure on ${\mathbb R}^n$.
We define \begin{equation*}M_{\mu}(v):= \int_{{\mathbb R}^n}
e^{\langle v,x\rangle }d\mu(x)=\exp (\Lambda_{\mu
}(v))\end{equation*}
where
\begin{equation*}\Lambda_{\mu }(v)=\ln \left (\int_{{\mathbb R}^n}
e^{\langle v,x\rangle }d\mu(x)\right )\end{equation*}is the logarithmic Laplace
transform of $\mu $. We also define
\begin{equation*}\Lambda_{\mu }^{\ast }(v):= {\cal L}(\Lambda_{\mu})(v)
= \sup_{u\in {\mathbb R}^n} \left\{ \langle v, u\rangle - \ln
\int_{{\mathbb R}^n} e^{\langle u,x\rangle
}d\mu(x)\right\},\end{equation*}
where, given a convex function $g:{\mathbb R}^n\to (-\infty ,\infty ]$, the Legendre transform ${\cal L}(g)$
of $g$ is defined by
\begin{equation*}{\cal L}(g)(x):=\sup_{y\in {\mathbb R}^n}\{ \langle x,y\rangle -
g(y)\}.\end{equation*}The function $\Lambda^{\ast }_{\mu
}$ is called the Cram\'{e}r transform of $\mu $ and plays a crucial
role in the theory of large deviations. For every $t\gr 1$ we define \begin{equation*}M_t(\mu) := \left\{
v\in {\mathbb R}^n : \int_{{\mathbb R}^n} | \langle v, x\rangle
|^td\mu(x)\ls 1\right\}.\end{equation*}Note that
\begin{equation*}Z_t(\mu) := (M_t(\mu))^{\circ}= \left\{
x\in{\mathbb R}^n : | \langle v, x\rangle |^t \ls \int_{{\mathbb
R}^n} | \langle v, y\rangle |^td\mu(y) \;\;\hbox{for all}\;
v\in{\mathbb R}^n\right\}.\end{equation*}For every $t>0$ we also set
\begin{equation*}B_t(\mu):=\{v\in{\mathbb R}^n:\Lambda^{\ast}_{\mu}(v)\ls t\}.\end{equation*}
We say that a measure $\mu $ on ${\mathbb R}^n$ is
$\alpha $-regular if for any $s\gr t\gr 2$
and every $v\in {\mathbb R}^n$,
\begin{equation*}
\left (\int_{{\mathbb R}^n} |\langle v,x\rangle |^sd\mu(x)\right
)^{1/s} \ls \alpha\frac{s}{t}\left (\int_{{\mathbb R}^n} |\langle
v,x\rangle |^td\mu(x)\right )^{1/t}.\end{equation*}
For all $s\gr t$ we have $M_s(\mu)\subseteq M_t(\mu)$
and $Z_t(\mu) \subseteq Z_s(\mu)$. If the measure $\mu$ is
$\alpha$-regular, then $M_t(\mu) \subseteq \alpha\frac{s}{t}
M_s(\mu)$ and $Z_s(\mu) \subseteq \alpha\frac{s}{t}Z_t(\mu)$ for all
$s\gr t\gr 2$. Moreover, for every centered probability measure $\mu $ we have
$\Lambda^{\ast }_{\mu }(0)=0$ by Jensen's inequality, and the convexity of $\Lambda^{\ast
}_{\mu }$ implies that $B_t(\mu) \subseteq B_s(\mu) \subseteq
\frac{s}{t}B_t(\mu)$ for all $s\gr t>0$.

Recall that, by Borell's lemma, every log-concave probability
measure is $c$-regular (see \cite[Theorem~2.4.6]{BGVV-book} for a proof).

\begin{proposition}\label{prop:fact-3}Every log-concave probability measure is $c$-regular, where
$c\gr 1$ is an absolute constant.
\end{proposition}

The next proposition compares $B_t(\mu )$ with $Z_t(\mu )$ when $\mu $ is $\alpha $-regular.

\begin{proposition}\label{prop:fact-4}If $\mu $ is $\alpha $-regular for some $\alpha\gr 1$, then for any $t\gr
2$ we have \begin{equation*}B_t(\mu) \subseteq 4e\alpha
Z_t(\mu).\end{equation*}
\end{proposition}

\begin{proof}We first check that if $u\in M_t(\mu)$ then
\begin{equation*}\Lambda_{\mu }\left ( \frac{tu}{2e\alpha } \right
)\ls t.\end{equation*} We fix $u\in M_t(\mu)$ and set
$\tilde{u}:=\frac{tu}{2e\alpha}$. Then, \begin{equation*}\left (
\int_{{\mathbb R}^n} |\langle\tilde{u},x\rangle |^kd\mu(x)\right
)^{1/k}= \frac{t}{2e\alpha}\left (\int_{{\mathbb R}^n} |\langle
u,x\rangle |^kd\mu(x)\right )^{1/k},\end{equation*} which is bounded
by $\frac{t}{2e\alpha}$ if $k\ls t$ and by $\frac{k}{2e}$ if $k>t$.
It follows that
\begin{align*}
\int_{{\mathbb R}^n} e^{\langle\tilde{u},x\rangle }d\mu(x) &\ls
\int_{{\mathbb R}^n} e^{|\langle\tilde{u},x\rangle |}d\mu(x) =
\sum_{k=0}^{\infty
}\frac{1}{k!}\int_{{\mathbb R}^n} |\langle\tilde{u}, x\rangle |^kd\mu(x)\\
&\ls \sum_{k\ls t}\frac{1}{k!} \left |\frac{t}{2e\alpha}\right |^k +
\sum_{k>t}\frac{1}{k!}\left |\frac{k}{2e}\right |^k \ls
e^{\frac{t}{2e\alpha}}+1\ls e^t\end{align*} and the claim follows.

Now, let $v\notin 4e\alpha Z_t(\mu)$. We can find $u\in M_t(\mu)$
such that $\langle v,u\rangle > 4e\alpha$ and then
\begin{equation*}\Lambda^{\ast }_{\mu }(v) \gr \Big\langle v, \frac{tu}{2e\alpha}\Big\rangle
-\Lambda_{\mu }\left ( \frac{tu}{2e\alpha}\right )>
\frac{t}{2e\alpha}4e\alpha -t=t.\end{equation*}Therefore, $v\notin
B_t(\mu )$. \end{proof}

\smallskip

By Proposition~\ref{prop:fact-3}, we have that Proposition~\ref{prop:fact-4} holds true (with an absolute constant in
place of $4e\alpha $) for every log-concave probability measure.

\smallskip

\noindent {\bf 2.5. Ball's bodies $K_t(\mu )$.} If $\mu$ is a log-concave probability
measure on $\mathbb R^n$ then, for every $t>0$, we define
\begin{equation*}
K_t(\mu):=K_t(f_\mu)=\left\{x\in {\mathbb R}^n : \int_0^\infty r^{t-1}f_\mu(rx)\,
dr\gr \frac{f_{\mu }(0)}{t} \right\}.
\end{equation*}
From the definition it follows that the radial function of $K_t(\mu )$ is given by
\begin{equation}\label{eq:radial-Kp}
\varrho_{K_t(\mu )}(x)=\left (\frac{1}{f_{\mu }(0)}\int_0^{\infty}tr^{t-1}f_{\mu }(rx)\,dr\right )^{1/t}\end{equation}
for $x\neq 0$. The bodies $K_t(\mu )$ were introduced by K.~Ball who also established their convexity. If $\mu $ is also centered then, for every $0 <
t\ls s$,
\begin{equation}\label{eq:inclusionsKp}\frac{\Gamma(t+1)^{\frac{1}{t}}}{\Gamma(s+1)^{\frac{1}{s}}} K_s(\mu )\subseteq
K_{t}(\mu )\subseteq e^{\frac{n}{t}-\frac{n}{s}}K_{s}(\mu ).
\end{equation} A proof is given in \cite[Proposition~2.5.7]{BGVV-book}. It is easily checked that
\begin{equation}\label{eq:volumeKn}|K_n(f)|\,f_{\mu }(0)=\int_{\mathbb R^n} f_{\mu }(x)dx=1\end{equation}
(see e.g. \cite[Lemma~2.5.6]{BGVV-book}) and then we can use the inclusions \eqref{eq:inclusionsKp} in order to
estimate the volume of $K_t(\mu )$. For every $t>0$ we have
\begin{equation}\label{eq:volumeKp}
e^{-1}\ls f_{\mu }(0)^{\frac{1}{n}+\frac{1}{t}}|K_{n+t}(\mu )|^{\frac{1}{n}+\frac{1}{t}}\ls
e\frac{n+t}{n}.
\end{equation}
We are mainly interested in the convex body $K_{n+1}(\mu )$. We shall use the fact that $K_{n+1}(\mu )$ is centered (see \cite[Proposition~2.5.3\,(v)]{BGVV-book}) and that
\begin{equation}\label{eq:volumeKn+1}f_{\mu }(0)|K_{n+1}(\mu )|\approx 1.\end{equation}
The last estimate follows immediately from \eqref{eq:volumeKn} and \eqref{eq:volumeKp}.

\section{Upper bound for the expected value of the half-space depth}\label{section-3}

Let $\mu $ and $\nu $ be two log-concave probability measures on ${\mathbb R}^n$ with the same barycenter.
If $T:{\mathbb R}^n\to {\mathbb R}^n$ is an invertible affine transformation and $T_{\ast}\mu $ is the push-forward
of $\mu $ defined by $T_{\ast}\mu (A)=\mu (T^{-1}(A))$ then
we observe that $\varphi_{T_{\ast }\mu }(x)=\varphi_{\mu }(T^{-1}(x))$ for all $x\in {\mathbb R}^n$, and hence
$$\int_{{\mathbb R}^n}\varphi_{T_{\ast }\mu }(x)dT_{\ast }\nu (x)=\int_{{\mathbb R}^n}\varphi_{\mu }(T^{-1}(x))dT_{\ast}\nu (x)
=\int_{{\mathbb R}^n}\varphi_{\mu }(x)d\nu (x).$$
Therefore, Theorem~\ref{th:intro-1} is a consequence of Theorem~\ref{th:intro-2}. Starting with a log-concave probability
measure $\mu $ on ${\mathbb R}^n$, we consider an affine transformation $T$ such that $T_{\ast }\mu $ is isotropic and
then apply Theorem~\ref{th:intro-2} to the measures $T_{\ast }\mu $ and $\nu =T_{\ast }\mu $.

\medskip

\begin{proof}[Proof of Theorem~$\ref{th:intro-2}$] Consider two isotropic log-concave probability measures $\mu, \nu $ on ${\mathbb R}^n$.
We will show that
$$\int_{\mathbb{R}^n}\varphi_{\mu }(x)\,d\nu(x) \ls e^{-cn/L_{\nu }^2}$$
for some absolute constant $c>0$.  We start with the observation that for any $v\in {\mathbb R}^n$ the half-space
$\{z:\langle z,v \rangle \gr \langle x,v\rangle\}$ is in ${\cal H}(x)$, therefore
\begin{equation*}
\varphi_{\mu } (x) \ls \mu (\{z:\langle z,v \rangle \gr \langle x,v\rangle\} )
\ls e^{-\langle x,v\rangle }{\mathbb E}_{\mu }\big(e^{\langle z,v\rangle }\big)=\exp \big(-[\langle x,v\rangle -\Lambda_{\mu }(v)]\big),\end{equation*}
and taking the infimum over all $v\in {\mathbb R}^n$ we see that
$$\varphi_{\mu } (x)\ls\exp (-\Lambda_{\mu }^{\ast }(x)).$$
Then we write
\begin{align*}\int_{\mathbb{R}^n}\varphi_{\mu } (x)\,d\nu(x) &\ls \int_{\mathbb{R}^n}e^{-\Lambda_{\mu }^{\ast } (x)}f_{\nu}(x)\, dx=\int_{\mathbb{R}^n}\left(\int_{\Lambda_{\mu }^{\ast } (x)}^{\infty }e^{-t}dt\,\right)f_{\nu}(x)dx\\
&=\int_0^{\infty }e^{-t}\int_{\mathbb{R}^n}{\bf{1}}_{B_t(\mu )}(x)f_{\nu}(x)dx\,dt=\int_0^{\infty }e^{-t}\nu (B_t(\mu ))\,dt.
\end{align*}
Fix $b\in (2/n,1/2]$ which will be chosen appropriately. Since $\nu (B_t(\mu ))\ls 1$ and also $\nu (B_t(\mu ))\ls\|f_{\nu }\|_{\infty }|B_t(\mu )|$
for all $t>0$, we may write
\begin{align*}
\int_{\mathbb{R}^n}\varphi_{\mu }(x)\,d\nu(x) &\ls\int_{bn}^{\infty }e^{-t}\nu (B_t(\mu ))dt +\|f_{\nu}\|_{\infty}\int_0^{bn}e^{-t}|B_t(\mu )|\,dt\\
&\ls\int_{bn}^{\infty }e^{-t}\,dt+L_{\nu }^n\int_0^2e^{-t}|B_t(\mu )|\,dt+L_{\nu }^n\int_2^{bn}e^{-t}|B_t(\mu )|\,dt \\
&\ls e^{-bn}+L_{\nu }^n|B_2(\mu )|+L_{\nu }^n\int_2^{bn}e^{-t}|B_t(\mu )|\,dt.
\end{align*}
Applying Proposition~\ref{prop:fact-4} and Theorem~\ref{th:fact-1} we get
$$|B_t(\mu )|^{1/n}\ls c_1|Z_t(\mu)|^{1/n}\ls c_2\sqrt{t/n}$$
for all $2\ls t\ls n$, where $c_1,c_2>0$ are absolute constants. It is also known that $L_{\nu }\gr c_3$
where $c_3>0$ is an absolute constant (see \cite[Proposition~2.3.12]{BGVV-book} for a proof). So, we
may assume that $c_2L_{\nu }\gr \sqrt{2}$. Choosing $b_0:=1/(c_2L_{\nu })^2\ls 1/2$ we write
$$L_{\nu }^n\int_2^{b_0n}e^{-t}|B_t(\mu )|\,dt\ls c_2^nL_{\nu }^n\int_2^{b_0n}(t/n)^{n/2}e^{-t}dt
=(c_2L_{\nu })^n\int_2^{b_0n}(t/n)^{n/2}e^{-t}dt,$$
and since $b_0n\ls n/2$ and the function $t\mapsto t^{n/2}e^{-t}$ is increasing on $[0,n/2]$, we get
$$(c_2L_{\nu })^n\int_2^{b_0n}e^{-t}|B_t(\mu )|\,dt\ls (b_0n-2)\cdot (c_2L_{\nu })^nb_0^{n/2}e^{-b_0n}=(b_0n-2)e^{-b_0n}.$$
Moreover, $|B_2(\mu )|^{1/n}\ls c_2\sqrt{2/n}$, therefore
$$L_{\nu }^n|B_2(\mu )|\ls (c_4L_{\nu }^2/n)^{n/2}\ls e^{-b_0n},$$
because $c_4L_{\nu }^2/n\ls e^{-2}$ if $n\gr n_0$. Combining the above we get
$$\int_{\mathbb{R}^n}\varphi_{\mu } (x)\,d\nu(x)\ls e^{-b_0n}+e^{-b_0n}+(b_0n-2)e^{-b_0n},$$
and hence
$$\int_{\mathbb{R}^n}\varphi_{\mu }(x)\,d\nu(x) \ls n\exp\left(-n/(c_2L_{\nu })^2\right)$$
which implies the result. \end{proof}

\begin{remark}\label{rem:variant}\rm In the introduction we have already mentioned that the assumption that
both $\mu $ and $\nu $ are isotropic is not necessary. One may consider different situations, where $\mu $
and $\nu $ are centered and $\|f_{\nu }\|_{\infty }$ is comparable with $\|f_{\mu }\|_{\infty }$. For example,
the next result can be obtained with the ideas that were used in the proof of Theorem~\ref{th:intro-2}.
\end{remark}

\begin{theorem}\label{th:variant}Let $\mu$ and $\nu $ be two centered log-concave probability measures on ${\mathbb R}^n$, $n\gr n_0$,
such that $\|f_{\mu}\|_{\infty }=\|f_{\nu }\|_{\infty }$. Then,
$${\mathbb E}_{\nu }(\varphi_{\mu }):=\int_{\mathbb{R}^n}\varphi_{\mu }(x)\,d\nu(x)\ls \exp\left(-cn/L_{\mu }^2\right),$$
where $c>0$, $n_0\in {\mathbb N}$ are absolute constants.
\end{theorem}

The proof of Theorem~\ref{th:variant} is quite similar to the one of Theorem~\ref{th:intro-2}. We fix $b\in (2/n,1/2]$ and write
\begin{align*}\int_{\mathbb{R}^n}\varphi_{\mu } (x)\,d\nu(x) &\ls \int_{\mathbb{R}^n}e^{-\Lambda_{\mu }^{\ast } (x)}f_{\nu}(x)\, dx
=\int_0^{\infty }e^{-t}\nu (B_t(\mu ))\,dt\\
&\ls e^{-bn}+\|f_{\nu }\|_{\infty }|B_2(\mu )|+\|f_{\nu }\|_{\infty }\int_2^{bn}e^{-t}|B_t(\mu )|\,dt.
\end{align*}
Then, we use the upper bound
$$|B_t(\mu )|^{1/n}\ls c_1|Z_t(\mu)|^{1/n}\ls c_2\sqrt{t/n}[\det \textrm{Cov}(\mu)]^{\frac{1}{2n}},$$
observe that
$$\|f_{\nu }\|_{\infty }[\det \textrm{Cov}(\mu)]^{\frac{1}{2}}=\|f_{\mu }\|_{\infty }[\det \textrm{Cov}(\mu)]^{\frac{1}{2}}=L_{\mu }^n.$$
and continue as in the proof of Theorem~\ref{th:intro-2}.

\section{Lower bound for the expected value of the half-space depth}\label{section-4}

In this section we show that the exponential estimate of Theorem~\ref{th:intro-1} is sharp. As explained in the introduction,
for the reader's convenience we consider first the simpler case where $\mu $ is the uniform measure on a convex body $K$ in
${\mathbb R}^n$ and then present the more technical tools and computations that are required for the case of an arbitrary log-concave probability measure $\mu $ on ${\mathbb R}^n$.

\subsection{Uniform measure on a convex body}

The next proposition provides an exponential lower bound for ${\mathbb E}_{\mu_K}(\varphi_{\mu_K})$, where $\mu_K$ is the
uniform measure on $K$.

\begin{proposition}\label{prop:body}Let $K$ be a convex body of volume $1$ in ${\mathbb R}^n$. Then,
$$\int_K\varphi_{\mu_K}(x)dx\gr e^{-cn},$$
where $c>0$ is an absolute constant.
\end{proposition}

\begin{proof}By translation invariance we may assume that the barycenter of $K$ is at the origin. Let $x\in\frac{1}{2}K$. We will show that
$\varphi_{\mu_K}(x)\gr \frac{1}{e^2n}\cdot\frac{1}{2^n}$. It suffices to show that
\begin{equation}\label{eq:body-1}\inf\,|\{z\in K:\langle z,\xi\rangle \gr\langle x,\xi\rangle\}|\gr\frac{1}{e^2n}\cdot\frac{1}{2^n},\end{equation}
where the infimum is over all $\xi\in S^{n-1}$, because by the definition of $\varphi_{\mu_K}(x)$ we only have to check the
half-spaces $H\in {\cal H}(x)$ for which $x$ is a boundary point. Moreover, we may consider only those $\xi\in S^{n-1}$
that satisfy $\langle x,\xi\rangle\gr 0$, because if
$\langle x,\xi\rangle <0$ then
$$|\{z\in K:\langle z,\xi\rangle \gr\langle x,\xi\rangle\}|\gr |\{z\in K:\langle z,\xi\rangle \gr 0\}|\gr 1/e$$
by Gr\"{u}nbaum's lemma (see \cite[Lemma~2.2.6]{BGVV-book}). Fix $\xi\in S^{n-1}$ with $\langle x,\xi\rangle\gr 0$
and set $m=h_K(\xi )=\max\{\langle z,\xi\rangle :z\in K\}$. Since $\langle x,\xi\rangle\ls
m/2$, it is enough to show that
\begin{equation}\label{eq:body-2}|\{z\in K:\langle z,\xi\rangle \gr m/2\}|\gr\frac{1}{e^2n}\cdot\frac{1}{2^n}.\end{equation}
Consider the function $g(t)=|K(\xi ,t)|_{n-1}$, where $K(\xi ,t)=\{z\in K:\langle z,\xi\rangle =t\}$, $t\in [0,m]$
and $|\cdot|_{n-1}$ denotes $(n-1)$-dimensional volume. The
Brunn-Minkowski inequality implies that $g^{\frac{1}{n-1}}$ is concave. Therefore, for every $r\in [0,m]$ we have that
$$g(r)\gr \left(1-\frac{r}{m}\right)^{n-1}g(0).$$
We write
\begin{align*}|\{z\in K:\langle z,\xi\rangle \gr m/2\}| &=\int_{m/2}^mg(r)\,dr\gr g(0)\int_{m/2}^m\left(1-\frac{r}{m}\right)^{n-1}dr\\
&= g(0)m\int_{1/2}^1(1-s)^{n-1}ds=\frac{1}{n2^n}g(0)m.
\end{align*}
Since $K$ is centered, we know that $\|g\|_{\infty }\ls e\,|K\cap\xi^{\perp }|_{n-1}=eg(0)$ from \eqref{eq:frad-1}. Then, using also Gr\"{u}nbaum's lemma,
we see that
$$\frac{1}{e}\ls \int_0^mg(r)\,dr\ls \|g\|_{\infty }m\ls eg(0)m,$$
and \eqref{eq:body-2} follows. It is now clear that
$$\int_K\varphi_{\mu_K}(x)dx\gr \int_{\frac{1}{2}K}\varphi_{\mu_K}(x)dx\gr \Big|\frac{1}{2}K\Big|\cdot \frac{1}{e^2n}\cdot\frac{1}{2^n}= \frac{1}{e^2n}\cdot\frac{1}{4^n}\gr e^{-cn}$$
for some absolute constant $c>0$.
\end{proof}

\subsection{Log-concave probability measures}

Next, we assume that $\mu $ is a log-concave probability measure on ${\mathbb R}^n$. Our aim is to prove
the next theorem.

\begin{theorem}\label{th:measure}Let $\mu $ be a log-concave probability measure on ${\mathbb R}^n$. Then,
$$\int_{{\mathbb R}^n}\varphi_{\mu}(x)d\mu(x)\gr e^{-cn},$$
where $c>0$ is an absolute constant.
\end{theorem}

By the affine invariance of ${\mathbb E}_{\mu }(\varphi_{\mu })$ we may assume that
$\mu $ is centered. The proof is based on a number of observations. The first one is a consequence of the Paley-Zygmund inequality; we just adapt here the proof of \cite[Lemma~11.3.3]{BGVV-book} to give a lower bound for $\varphi_{\mu }(x)$ when $x\in \delta Z_t^+(\mu )$ for some $\delta\in (0,1)$.

\begin{lemma}\label{lem:measure-1}Let $t\gr 1$ and $\delta\in (0,1)$. For every $x\in \delta Z_t^+(\mu )$ one
has \begin{equation*}\varphi_{\mu }(x)\gr
\frac{(1-\delta^t)^2}{C_1^t},\end{equation*}
where $C_1>1$ is an absolute constant.
\end{lemma}

\begin{proof}Let $x\in \delta Z_t^+(\mu )$. As in the proof of Proposition~\ref{prop:body}, using Gr\"{u}nbaum's lemma
we see that it is enough to show that
\begin{equation}\inf\mu (\{z\in {\mathbb R}^n:\langle z,\xi\rangle \gr\langle x,\xi\})\gr\frac{(1-\delta^t)^2}{C_1^t},\end{equation}
where the infimum is over all $\xi\in S^{n-1}$ with $\langle x,\xi\rangle\gr 0$.

Since $x\in \delta Z_t^+(\mu )$, we have $\langle x,\xi\rangle\ls \delta h_{Z_t^+(\mu )}(\xi )$ for any such $\xi\in S^{n-1}$,
so it is enough to show that
\begin{equation}\mu(\{z\in {\mathbb R}^n:\langle z,\xi\rangle \gr \delta h_{Z_t^+(\mu )}(\xi)\})\gr\frac{(1-\delta^t)^2}{C_1^t}.\end{equation}
We apply the Paley-Zygmund inequality
\begin{equation*}\mu (\{z:g(z)\gr \delta^t{\mathbb E}_{\mu }(g)\})\gr
(1-\delta^t)^2\frac{[{\mathbb E}_{\mu }(g)]^2}{{\mathbb E}_{\mu }(g^2)}
\end{equation*} for the function $g(z)=\langle z,\xi\rangle_+^t$. From \eqref{eq:pz-1} we see that
\begin{equation*}{\mathbb E}_{\mu }(g^2)\ls C_1^t\,[{\mathbb E}_{\mu }(g)]^2 \end{equation*} for
some absolute constant $C_1>0$, and the lemma follows. \end{proof}

\begin{definition}\rm For every $t\gr 1$ we consider the convex set
$$R_t(\mu )=\{x\in {\mathbb R}^n:f_{\mu }(x)\gr e^{-t}f_{\mu }(0)\}.$$
The convexity of $R_t(\mu )$ is an immediate consequence of the log-concavity of $f_{\mu }$. Note that
$R_t(\mu )$ is bounded and $0\in {\rm int}(R_t(\mu ))$.
\end{definition}

\begin{lemma}\label{lem:measure-4}For every $t\gr 5n$ we have $R_t(\mu )\supseteq c_0K_{n+1}(\mu )$, where $c_0>0$ is an
absolute constant.
\end{lemma}

\begin{proof}Let $t\gr 5n$. Given any $\xi\in S^{n-1}$ consider the log-concave function $h:[0,\infty )\to [0,\infty )$ defined by
$h(t)=f_{\mu }(t\xi )$. From \cite[Lemma~5.2]{Klartag-2007clt} we know that
$$\int_0^{\varrho_{R_t(\mu )}(\xi )}r^{n-1}h(r)dr\gr (1-e^{-t/8})\int_0^{\infty }r^{n-1}h(r)dr.$$
By the definition of $K_n(\mu )$ we have
$$\int_0^{\infty }r^{n-1}h(r)dr=\frac{f_{\mu }(0)}{n}[\varrho_{K_n(\mu )}(\xi )]^n.$$
On the other hand,
$$\int_0^{\varrho_{R_t(\mu )}(\xi )}r^{n-1}h(r)dr\ls \|f\|_{\infty }\int_0^{\varrho_{R_t(\mu )}(\xi )}r^{n-1}dr
=\frac{\|f\|_{\infty }}{n}[\varrho_{R_t(\mu )}(\xi )]^n.$$
Using also the fact that $\|f\|_{\infty }\ls e^nf_{\mu }(0)$ from \eqref{eq:frad-2}, we get
$$e^n[\varrho_{R_t(\mu )}(\xi )]^n\gr (1-e^{-t/8})[\varrho_{K_n(\mu )}(\xi )]^n.$$
This shows that $R_t(\mu )\supseteq c_0K_n(\mu )$, where $c_0>0$ is an
absolute constant. From \eqref{eq:inclusionsKp} we know that $K_n(\mu )\approx K_{n+1}(\mu)$,
and this completes the proof.\end{proof}

\smallskip

Our final lemma compares $Z_t^+(\mu )$ with $K_{n+1}(\mu )$ when $t\gr 5n$.

\begin{lemma}\label{lem:measure-5}For every $t\gr 5n$ we have that $Z_t^+(\mu )\supseteq c_0^{\prime }K_{n+1}(\mu )$, where $c_0^{\prime }>0$ is an
absolute constant.
\end{lemma}

\begin{proof} From Lemma~\ref{lem:measure-4} we know that $c_0K_{n+1}(\mu )\subseteq R_t(\mu )$ for all $t\gr 5n$,
where $c_0>0$ is an absolute constant. Let $\xi\in S^{n-1}$ and
set $m:=h_{c_0K_{n+1}(\mu )}(\xi )=c_0h_{K_{n+1}(\mu )}(\xi )$. Define
$$A_{\xi }=c_0K_{n+1}(\mu )\cap \{ x:\langle x,\xi\rangle\gr m/2\}.$$
Since $K_{n+1}(\mu )$ is centered, the proof of Proposition~\ref{prop:body} shows that
$$|A_{\xi}|\gr \frac{|c_0K_{n+1}(\mu)|}{e^2n\cdot 2^n}\gr \frac{|c_0K_{n+1}(\mu)|}{C^n}$$
for some absolute constant $C>c_0$. Moreover, if $x\in A_{\xi }$ then $x\in R_t(\mu )$ and hence $f_{\mu }(x)\gr e^{-t}f_{\mu }(0)$. We write
\begin{align*}
\int_{{\mathbb R}^n}\langle x,\xi\rangle_+^td\mu (x) &\gr \int_{A_{\xi }}\langle x,\xi\rangle_+^td\mu (x)\\
&\gr \left(\frac{m}{2}\right)^te^{-t}f_{\mu }(0)|A_{\xi }|\gr \left(\frac{m}{2e}\right)^t\left(\frac{c_0}{C}\right)^nf_{\mu }(0)|K_{n+1}(\mu )|.
\end{align*}
Using also the fact that $(c_0/C)^n\gr (c_0/C)^t$ because $t\gr 5n$, we get
$$\int_{{\mathbb R}^n}\langle x,\xi\rangle_+^td\mu (x) \gr (c_1m)^tf_{\mu }(0)|K_{n+1}(\mu )|,$$
where $c_1>0$ is an absolute constant. Finally, $f_{\mu }(0)|K_{n+1}(\mu )|\approx 1$ by \eqref{eq:volumeKn+1}, which implies that
$$h_{Z_t^+(\mu )}(\xi )\gr c_2m=c_0^{\prime }h_{K_{n+1}(\mu )}(\xi ),$$
where $c_0^{\prime }=c_2c_0$, and the lemma is proved. \end{proof}

\medskip

\begin{proof}[Proof of Theorem~$\ref{th:measure}$]Combining Lemma~\ref{lem:measure-4} and Lemma~\ref{lem:measure-5}
we see that
$$R_{5n}(\mu)\cap Z_{5n}^+(\mu )\supseteq c_1K_{n+1}(\mu )$$
for some absolute constant $c_1>0$. We apply Lemma~\ref{lem:measure-1} with $t=5n$ and $\delta =\frac{1}{2}$. For every $x\in \frac{1}{2}Z_{5n}^+(\mu )$
we have \begin{equation*}\varphi_{\mu }(x)\gr C_1^{-n}\end{equation*}
for some absolute constant $C_1>1$. It follows that
$$\int_{{\mathbb R}^n}\varphi_{\mu }(x)\,d\mu (x)\gr C_1^{-n}\mu \left(\tfrac{1}{2}Z_{5n}^+(\mu )\right ).$$
Then, by Lemma~\ref{lem:measure-5} we have $\frac{1}{2}Z_{5n}^+(\mu )\supseteq \frac{c_1}{2}K_{n+1}(\mu )$. Since
$\frac{c_1}{2}K_{n+1}(\mu )\subseteq R_{5n}(\mu)$, we know that $f_{\mu }(x)\gr e^{-5n}f_{\mu }(0)$ for all
$x\in \frac{c_1}{2}K_{n+1}(\mu )$. Using also \eqref{eq:volumeKn+1}, we get
\begin{align*}\mu \left(\tfrac{1}{2}Z_{5n}^+(\mu )\right) &\gr \mu \left(\frac{c_1}{2}K_{n+1}(\mu )\right)=\int_{\frac{c_1}{2}K_{n+1}(\mu )}f_{\mu }(x)\,dx\gr e^{-5n}f_{\mu }(0)\left|\frac{c_1}{2}K_{n+1}(\mu )\right|\\
&= e^{-5n}(c_1/2)^nf_{\mu }(0)|K_{n+1}(\mu )|\gr e^{-5n}c_2^n.
\end{align*}
Combining the above we conclude that
$$\int_{{\mathbb R}^n}\varphi_{\mu }(x)\,d\mu (x)\gr C_1^{-n}e^{-5n}c_2^n\gr e^{-cn},$$
for some absolute constant $c>0$.
\end{proof}

\section{Bounds for the expected measure of random polytopes}\label{section-5}

Let $\mu $ be a log-concave probability measure on ${\mathbb R}^n$. Let $X_1,X_2,\ldots $ be independent random points in
${\mathbb R}^n$ distributed according to $\mu $ and for any $N>n$ consider the random polytope $K_N={\rm conv}\{X_1,\ldots ,X_N\}$.
Given a second log-concave probability measure $\nu $ on ${\mathbb R}^n$ with the same barycenter as $\mu $, consider
the expectation ${\mathbb E}_{\mu^N}[\nu (K_N)]$ of the $\nu $-measure of $K_N$. Note that if $T:{\mathbb R}^n\to {\mathbb R}^n$
is an invertible affine transformation and $T_{\ast}\mu $ is the push-forward
of $\mu $ defined by $T_{\ast}\mu (A)=\mu (T^{-1}(A))$ then
$${\mathbb E}_{(T_{\ast }\mu)^N}[(T_{\ast }\nu)(K_N)]={\mathbb E}_{\mu^N}[\nu (K_N)].$$
So, we may assume that $\mu $ is isotropic and $\nu $ is centered. In the
next theorem we actually assume that both $\mu $ and $\nu $ are isotropic.

\begin{theorem}\label{th:upper-threshold}
Let $\mu$ and $\nu $ be isotropic log-concave probability measures on ${\mathbb R}^n$, $n\gr n_0$. For any
$N\ls \exp (c_1n/L_{\nu }^2)$ we have that
$${\mathbb E}_{\mu^N}(\nu (K_N)) \ls 2 \exp\left(-c_2n/L_{\nu }^2\right),$$
where $c_1,c_2>0$ and $n_0\in {\mathbb N}$ are absolute constants.
\end{theorem}

The proof of Theorem~\ref{th:upper-threshold} will exploit the same tools that were used in the previous section, combined with
a variant of the standard lemma that is used in order to establish upper thresholds.
Recall that $B_t(\mu)=\{v\in{\mathbb R}^n:\Lambda^{\ast}_{\mu}(v)\ls t\}$, where $\Lambda^{\ast }_{\mu
}$ is the Cram\'{e}r transform of $\mu $. A version of the next lemma appeared
originally in \cite{DFM}.

\begin{lemma}\label{lem:upper-1}Let $t>0$. For every $N>n$ we have
$${\mathbb E}_{\mu^N}(\nu (K_N))\ls \nu (B_t(\mu ))+ N\exp (-t).$$
\end{lemma}

\begin{proof}We write
\begin{align*}{\mathbb E}_{\mu^N}(\nu (K_N)) &={\mathbb E}_{\mu^N}(\nu (K_N\cap B_t(\mu)))+{\mathbb E}_{\mu^N}(\nu (K_N\setminus B_t(\mu )))\\
&\ls \nu (B_t(\mu))+{\mathbb E}_{\mu^N}(\nu (K_N\setminus B_t(\mu ))).\end{align*}
Observe that if $H$ is a closed half-space containing $x$, and if $x\in K_N$, then there exists
$i\ls N$ such that $X_i\in H$ (otherwise we would have $x\in K_N\subseteq H^\prime$, where $H^\prime$ is the
complementary half-space). It follows that
$$\mu^N\bigl(x\in K_N\bigr) \ls N\varphi_{\mu }(x).$$
Then, Fubini's theorem shows that
$${\mathbb E}_{\mu^N}(\nu (K_N\setminus B_t(\mu )))=\int_{{\mathbb R}^n\setminus
B_t(\mu )}\mu^N(x\in K_N)\,d\nu (x)\ls
\int_{{\mathbb R}^n\setminus B_t(\mu )}N\varphi_{\mu }(x)\,d\nu (x)\ls N\,e^{-t}$$
because $\varphi_{\mu }(x)\ls \exp(-\Lambda_{\mu}^{\ast }(x))\ls e^{-t}$ for all
$x\notin B_t(\mu )$.\end{proof}

\smallskip

\begin{proof}[Proof of Theorem~$\ref{th:upper-threshold}$]
Using the estimate $\nu (B_t(\mu ))\ls \|f_{\nu }\|_{\infty }|B_t(\mu )|$, Proposition~\ref{prop:fact-4}
and Theorem~\ref{th:fact-1}, from Lemma~\ref{lem:upper-1} we get
$${\mathbb E}_{\mu^N}(\nu (K_N))\ls \left(c_1\|f_{\nu }\|_{\infty }^{1/n}\sqrt{t/n}\right)^n+ N\exp (-t)$$
for every $N>n$ and $2\ls t\ls n$. Recall that $\nu $ is isotropic, therefore $\|f_{\nu }\|_{\infty }^{2/n}=L_{\nu }^2=O(\sqrt{n})$;
in fact, Klartag's estimate for $L_n$ gives much more. Then, if $n\gr n_0$ where $n_0\in{\mathbb N}$ is an absolute constant,
the choice $t:=(c_1e)^{-2}n/\|f_{\nu }\|_{\infty }^{2/n}$ satisfies $2\ls t\ls n$ and gives
$$\left(c_1\|f_{\nu }\|_{\infty }^{1/n}\sqrt{t/n}\right)^n\ls e^{-n}.$$ Therefore,
$${\mathbb E}_{\mu^N}(\nu (K_N))\ls e^{-n}+ N\exp (-c_2n/\|f_{\nu }\|_{\infty }^{2/n}),$$
where $c_2=(c_1e)^{-2}$. It follows that if $N\ls \exp (c_3n/\|f_{\nu }\|_{\infty }^{2/n})$ where $c_3=c_2/2$, then we have
$${\mathbb E}_{\mu^N}(\nu (K_N))\ls e^{-n}+\exp (-c_3n/\|f_{\nu }\|_{\infty }^{2/n})$$
and the result follows from the fact that $\|f_{\nu }\|_{\infty }^{2/n}=L_{\nu }^2\gr c$. \end{proof}

\begin{remark}\label{rem:variants-5}\rm Let $\mu $ be isotropic. For the proof
of Theorem~\ref{th:upper-threshold}, what we actually need about $\nu $ is that $\nu $ is centered and
that $\|f_{\nu }\|_{\infty }^{1/n}=o_n(\sqrt{n})$. Then the argument of the previous proof
gives
$${\mathbb E}_{\mu^N}(\nu (K_N))\ls \exp (-c_2n/\max\{1,\|f_{\nu }\|_{\infty }^{2/n}\})$$
if $N\ls \exp (c_1n/\|f_{\nu }\|_{\infty }^{2/n})$. Note that the proof of \eqref{eq:tk-1}
in \cite{Chakraborti-Tkocz-Vritsiou-2021} exploits the same ideas. The role of $\nu $ is played
by the uniform measure on a convex body $K$, therefore $\|f_{\nu }\|_{\infty }=\frac{1}{|K|}$.
On the other hand, $\mu $ is isotropic and supported on $K$, and hence
$$|K|\cdot L_{\mu }^n\gr \int_Kf_{\mu }(x)dx=\mu (K)=1.$$
This shows that $\|f_{\nu }\|_{\infty }\ls L_{\mu }^n$, therefore $n/\|f_{\nu }\|_{\infty}^{\frac{2}{n}}\gr n/L_{\mu }^2$,
which (combined with the above) proves \eqref{eq:tk-1}.

A second possible normalization is to assume that $\mu $ and $\nu $ are simply centered and that $\|f_{\mu }\|_{\infty }=
\|f_{\nu}\|_{\infty }$. Then, starting the computation as in the proof of Theorem~\ref{th:upper-threshold} we get
\begin{align*}
{\mathbb E}_{\mu^N}(\nu (K_N)) &\ls \left(c_1\|f_{\nu }\|_{\infty }^{1/n}[\det {\rm Cov}(\mu)]^{\frac{1}{2n}}\sqrt{t/n}\right)^n+ N\exp (-t)\\
&= \left(c_1\|f_{\mu }\|_{\infty }^{1/n}[\det {\rm Cov}(\mu)]^{\frac{1}{2n}}\sqrt{t/n}\right)^n+ N\exp (-t)=
\left(c_1L_{\mu }\sqrt{t/n}\right)^n+ N\exp (-t).
\end{align*}
Choosing $t=(c_1e)^{-2}n/L_{\mu }^2$ and continuing as above, we get:
\end{remark}

\begin{theorem}\label{th:upper-threshold-variant}
Let $\mu$ and $\nu $ be two centered log-concave probability measures on ${\mathbb R}^n$
with $\|f_{\mu }\|_{\infty }=\|f_{\nu }\|_{\infty }$. For any
$N\ls \exp (c_1n/L_{\mu }^2)$ we have that
$${\mathbb E}_{\mu^N}(\nu (K_N)) \ls 2 \exp\left(-c_2n/L_{\mu }^2\right),$$
where $c_1,c_2>0$ are absolute constants.
\end{theorem}

We pass now to the lower threshold. It is useful to observe that in the case where $X_1,X_2,\ldots $
are uniformly distributed in the Euclidean unit ball the sharp threshold for the problem (see \cite{Pivovarov-2007} and \cite{Bonnet-Chasapis-Grote-Temesvari-Turchi-2019}) is
$$\exp \left((1\pm \epsilon )\tfrac{1}{2}n\ln n\right),\qquad \epsilon >0.$$
We concentrate in the case $\nu =\mu $ of our problem, in which case we shall establish a weak lower threshold of this order.
The precise formulation of our result is the following.

\begin{theorem}\label{th:lower-threshold}Let $\delta\in (0,1)$. Then,
$$\inf_{\mu }\Big(\inf\Big\{ {\mathbb E}_{\mu^N}\big[\mu ((1+\delta )K_N)\big]: N\gr \exp \big (C\delta^{-1}\ln \left(2/\delta \right)n\ln n\big )\Big\}\Big)\longrightarrow 1$$ as $n\to\infty $, where the first infimum is over all centered log-concave probability measures $\mu $ on ${\mathbb R}^n$ and $C>0$
is an absolute constant.
\end{theorem}

This is a weak threshold in the sense that we consider the expected measure of $(1+\delta )K_N$ instead of $K_N$, where
$\delta >0$ is arbitrarily small. The reason for this is the dependence on $\delta $ in the next technical
proposition.

\begin{proposition}\label{prop:muZ}Let $\mu $ be an isotropic log-concave probability measure on ${\mathbb R}^n$. For any $\delta\in (0,1)$
and any $t\gr C_{\delta }n\ln n$ we have that $$\mu ((1+\delta )Z_t^+(\mu ))\gr 1-e^{-c_{\delta }t}$$
where $C_{\delta }=C\delta^{-1}\ln\left(2/\delta\right)$ and $c_{\delta }=c\delta $ are positive
constants depending only on $\delta $.
\end{proposition}

\begin{proof}Let $\delta\in (0,1)$ and set $\epsilon =\delta /5$. Fix $t\gr n$ which will be determined. Recall that $b_1B_2^n\subseteq Z_t^+(\mu )\subseteq b_2tB_2^n$ for some absolute constants $b_1,b_2>0$. This implies that if $v,w\in S^{n-1}$ and $|v-w|\ls \frac{b_1\epsilon }{b_2t}$ then
$$h_{Z_t^+(\mu )}(v-w)\ls b_2t|v-w|\quad\hbox{and}\quad b_1\ls \min\{h_{Z_t^+(\mu )}(v),h_{Z_t^+(\mu )}(w)\},$$
therefore
\begin{equation}\label{muZ-1}h_{Z_t^+(\mu )}(v-w)\ls b_2t|v-w|\ls \epsilon\min\{h_{Z_t^+(\mu )}(v),h_{Z_t^+(\mu )}(w)\}.\end{equation}
Set $b:=b_2/b_1$ and consider a $\frac{\epsilon }{bt}$-net $N$ of the Euclidean unit sphere $S^{n-1}$ with  cardinality $|N|\ls (1+2bt/\epsilon )^n\ls (3bt/\epsilon )^n$; for a proof of the estimate on the cardinality of $N$ see e.g. \cite[Lemma~5.2.5]{AGA-book}.
We define
$$W=\bigcap_{\xi\in N}\left\{x:\langle x,\xi\rangle_+\ls \frac{1}{1+\epsilon } h_{Z_t^+(\mu )}(\xi )\right\}.$$
Let $x\in W$.
Then, $\langle x,\xi\rangle_+\ls\frac{1}{1+\epsilon } h_{Z_t^+(\mu )}(\xi )$ for all $\xi\in N$.
We will show that
$(1-\epsilon )\langle x,w\rangle_+\ls h_{Z_t^+(\mu )}(w)$ for all $w\in S^{n-1}$, which is equivalent
to $(1-\epsilon )x\in Z_t^+(\mu )$. We set
$$\alpha_{\mu }(x):=\max\left\{\frac{\langle x,w\rangle_+}{h_{Z_t^+(\mu )}(w)}:w\in S^{n-1}\right\}$$
and consider $v\in S^{n-1}$ such that $\langle x,v\rangle_+=\alpha_{\mu }(x)\cdot h_{Z_t^+(\mu )}(v)$.
There exists $\xi\in N$ such that $|\xi -v|\ls \frac{\epsilon }{bt}$. Using the fact that
$\langle x,v-\xi\rangle_+\ls \alpha_{\mu }(x)h_{Z_t^+(\mu )}(v-\xi )$, we write
$$\langle x,v\rangle_+ \ls \langle x,\xi \rangle_+ +\langle x,v-\xi\rangle_+ \ls \frac{1}{1+\epsilon } h_{Z_t^+(\mu )}(\xi )
+\alpha_{\mu }(x)h_{Z_t^+(\mu )}(v-\xi ).$$
From \eqref{muZ-1} it follows that
$$\langle x,v\rangle_+ \ls \frac{1}{1+\epsilon } h_{Z_t^+(\mu )}(\xi )+\epsilon \alpha_{\mu }(x)h_{Z_t^+(\mu )}(v)=
\frac{1}{1+\epsilon } h_{Z_t^+(\mu )}(\xi )+\epsilon \langle x,v\rangle_+,$$
which gives $$\langle x,v\rangle_+\ls \frac{1}{1-\epsilon^2} h_{Z_t^+(\mu )}(\xi ).$$
Moreover,
$$h_{Z_t^+(\mu )}(\xi )\ls h_{Z_t^+(\mu )}(v)+h_{Z_t^+(\mu )}(\xi -v)\ls h_{Z_t^+(\mu )}(v)+\epsilon h_{Z_t^+(\mu )}(v)
=(1+\epsilon )h_{Z_t^+(\mu )}(v),$$
which finally gives $\alpha_{\mu }(x)\ls 1/(1-\epsilon )$. This shows that $(1-\epsilon )W\subseteq Z_t^+(\mu )$.
For every $\xi\in N$ we have
$$\mu (\{x:\langle x,\xi\rangle_+\gr (1+\epsilon )\|\langle \cdot ,\xi\rangle_+ \|_t\})\ls (1+\epsilon )^{-t}.$$
Since $\delta\in (0,1)$ we have $0<\epsilon <1/5$, therefore $\frac{(1+\epsilon )^2}{1-\epsilon }\ls 1+5\epsilon =1+\delta $. Then,
\begin{align*}\mu ((1+\delta )Z_t^+(\mu )) &\gr \mu \left (\frac{(1+\epsilon)^2}{1-\epsilon }Z_t^+(\mu )\right )\gr \mu ((1+\epsilon)^2W)\\
&=\mu\left(\bigcap_{\xi\in N}\left\{x:\langle x,\xi\rangle_+\ls (1+\epsilon ) h_{Z_t^+(\mu )}(\xi )\right\}\right)\\
&\gr 1-|N|\cdot (1+\epsilon )^{-t}\gr 1-(C^{\prime }_{\epsilon }t)^n(1+\epsilon )^{-t},\end{align*}
where $C_{\epsilon }^{\prime }=3b/\epsilon $. It follows that there exists $C_{\epsilon }>1$ such that if $t\gr C_{\epsilon }n\ln n$ then
\begin{equation}\label{eq:Cepsilon}(C_{\epsilon }^{\prime }t)^n(1+\epsilon )^{-t}\ls (1+\epsilon )^{-t/2}\ls e^{-\epsilon t/4}.\end{equation}
To see this, consider the function
$$\ell (t)=\frac{t}{2}\ln (1+\epsilon )-n\ln (3bt/\epsilon ).$$
It is easily checked that $\ell $ is increasing on $[2n/\ln (1+\epsilon ),\infty )$. Therefore, if $t\gr C_{\epsilon }n\ln n$
where $C_{\epsilon }=\frac{C}{\epsilon }\ln\left(\frac{2}{\epsilon }\right)$ for a large enough absolute constant $C>0$,
one can check that $\ell (t)\gr \ell (C_{\epsilon }n\ln n)>0$. This implies \eqref{eq:Cepsilon}.
Since $\epsilon =\delta /5$, we obtain the assertion of the proposition with the stated dependence
of the constants $C_{\delta },c_{\delta }$ on $\delta $.
\end{proof}

\smallskip

For the proof of Theorem~\ref{th:lower-threshold} we also need a basic fact that plays a main role in the proof of all
the lower thresholds that have been obtained so far. It is stated in the form below in \cite[Lemma~3]{Chakraborti-Tkocz-Vritsiou-2021}.
For a proof see \cite{DFM} or \cite[Lemma~4.1]{Gatzouras-Giannopoulos-2009}.

\begin{lemma}\label{lem:inclusion}For every Borel subset $A$ of ${\mathbb R}^n$ we have that
$$1-\mu^N(K_N\supseteq A)\ls 2\binom{N}{n}\left (1-\inf_{x\in A}\varphi_{\mu }(x)\right)^{N-n}.$$
Therefore,
$${\mathbb E}_{\mu^N}[\mu (K_N)]\gr \mu (A)\left (1-2\binom{N}{n}\left (1-\inf_{x\in A}\varphi_{\mu }(x)\right)^{N-n}\right).$$
\end{lemma}

\begin{proof}[Proof of Theorem~$\ref{th:lower-threshold}$]Let $0<\delta <1$ and set $\epsilon =\delta /3$.
Let $\mu$ be a centered log-concave probability
measure on ${\mathbb R}^n$. Since the expectation ${\mathbb E}_{\mu^N}\big[\mu ((1+\delta )K_N)\big]$
is a linearly invariant quantity, we may assume that $\mu$ is isotropic. 
From Lemma~\ref{lem:measure-1} we know that for every $x\in (1-\epsilon) Z_t^+(\mu )$ we have
\begin{equation*}\varphi_{\mu }(x)\gr\frac{(1-(1-\epsilon)^t)^2}{C_1^t},\end{equation*}
where $C_1>1$ is an absolute constant. Then, taking into account the fact that $1-\epsilon >2/3$, we get
\begin{equation*}\mu^N\Big(K_N\supseteq (1-\epsilon )Z_t^+(\mu )\Big)
\gr 1-2\binom{N}{n}\left [1-\frac{(1-(1-\epsilon)^t)^2}{C_1^t}\right]^{N-n}.\end{equation*}
By the mean value theorem we have $1-(1-\epsilon )^t=t\epsilon z^{t-1}$ for some $z\in (1-\epsilon ,1)$, and hence
$1-(1-\epsilon )^t\gr t\epsilon (1-\epsilon)^{t-1}$. Taking also into account the fact that $1-\epsilon >2/3$, we get
\begin{align*}\mu^N\Big(K_N\supseteq (1-\epsilon )Z_t^+(\mu )\Big)
&\gr 1-2\binom{N}{n}\left [1-\frac{({t\epsilon}(1-\epsilon)^{t-1})^2}{C_1^t}\right]^{N-n}\\
&\gr 1-\left(\frac{2eN}{n}\right)^n\exp \left(-(N-n)\frac{(t\epsilon )^2}{(3C_1)^t}\right).
\end{align*}
This last quantity tends to $1$ as $n\to\infty $ if
\begin{equation}\label{eq:conditionC1}(3C_1)^tn\ln (4eN/n)<(N-n)(t\epsilon )^2,\end{equation}
and assuming that $\delta\in (1/n^2,1)$ and $t\gr C_{\epsilon }n\ln n$ where $C_{\epsilon }$ is the constant from Proposition~\ref{prop:muZ},
we check that \eqref{eq:conditionC1} holds true if $N\gr \exp (C_2t)$ for a large enough absolute constant $C_2>0$.

Note that $\epsilon =\delta /3$ implies that $1+\delta >\frac{1+\epsilon }{1-\epsilon }$. Then, if $N\gr \exp (C_2C_{\epsilon }n\ln n)$
we see that
\begin{align*}{\mathbb E}_{\mu^N}\left[\mu \left((1+\delta )K_N\right)\right] &\gr  {\mathbb E}_{\mu^N}\left[\mu \left(\frac{1+\epsilon }{1-\epsilon }K_N\right)\right]
\gr \mu ((1+\epsilon )Z_t^+(\mu ))\times \mu^N\Big(K_N\supseteq (1-\epsilon )Z_t^+(\mu )\Big)\\
&\gr \big(1-e^{-c\epsilon t}\big)\left[1-\left(\frac{2eN}{n}\right)^n\exp \left(-(N-n)\frac{(t\epsilon )^2}{(3C_1)^t}\right)\right]\longrightarrow 1
\end{align*}
as $n\to\infty $.
\end{proof}

\smallskip

We have already mentioned that Theorem~\ref{th:lower-threshold} provides a weak threshold in the sense that we estimate the
expectation ${\mathbb E}_{\mu^N}\big(\mu ((1+\delta )K_N)\big)$ (for an arbitrarily small but positive value of $\delta )$ while
the original question is about ${\mathbb E}_{\mu^N}\big(\mu (K_N)\big)$. The next result provides an estimate where ``$\delta $ is removed", however the dependence on $n$ is worse. The argument below was suggested by the referee and replaces
our much more complicated original argument, leading to the same final estimate.

\begin{theorem}\label{th:non-sharp}There exists an absolute constant $C>0$ such that
$$\inf_{\mu }\Big(\inf\Big\{ {\mathbb E}_{\mu^N}\big[\mu (K_N)\big]: N\gr \exp (C(n\ln n)^2u(n))\Big\}\big)\longrightarrow 1$$
as $n\to\infty $, where the first infimum is over all log-concave probability measures $\mu $ on ${\mathbb R}^n$
and $u(n)$ is any function with $u(n)\to\infty $ as $n\to\infty $.
\end{theorem}

\begin{proof}Let $\mu$ be a log-concave probability measure on ${\mathbb R}^n$. Since the expectation ${\mathbb E}_{\mu^N}\big[\mu (K_N)\big]$ is an affinely invariant quantity, we may assume that $\mu$ is centered. Note that if $A\subset {\mathbb R}^n$
is a Borel set, then
$$\mu((1+\delta)A)=\int_{(1+\delta)A}f_{\mu}(x)\,dx=(1+\delta)^n\int_Af_{\mu}((1+\delta)x)\,dx.$$
Since $f_{\mu}$ is log-concave, we see that
$$f_{\mu}((1+\delta)x)\ls f_{\mu}(x)\left(\frac{f_{\mu}(x)}{f_{\mu}(0)}\right)^{\delta}\ls e^{n\delta}f_{\mu}(x)$$
for every $x\in {\mathbb R}^n$, because $f_{\mu}(x)\ls e^nf_{\mu}(0)$ by \eqref{eq:frad-2}. It follows that
\begin{equation}\label{eq:ref-1}\mu((1+\delta)A)\ls (1+\delta)^ne^{n\delta}\mu(A)\ls e^{2n\delta}\mu(A).\end{equation}
Given a function $u(n)$ with $u(n)\to\infty $ as $n\to\infty $, choose $\delta_n=(nu(n))^{-1}$. From \eqref{eq:ref-1}
we see that
$${\mathbb E}_{\mu^N}\big[\mu (K_N)\big]\gr e^{-2n\delta_n}{\mathbb E}_{\mu^N}\big[\mu ((1+\delta_n)K_N)\big].$$
Therefore, we see that
\begin{align*}
&\inf_{\mu }\Big(\inf\Big\{ {\mathbb E}_{\mu^N}\big[\mu (K_N)\big]: N\gr \exp \big (C\delta_n^{-1}\ln \left(2/\delta_n \right)n\ln n\big )\Big\}\Big)\\
&\hspace*{1cm}\gr e^{-2n\delta_n}\inf_{\mu }\Big(\inf\Big\{ {\mathbb E}_{\mu^N}\big[\mu ((1+\delta_n)K_N)\big]: N\gr \exp \big (C\delta_n^{-1}\ln \left(2/\delta_n \right)n\ln n\big )\Big\}\Big)\longrightarrow 1
\end{align*}
as $n\to\infty $, using Theorem~\ref{th:lower-threshold} and the fact that $e^{-2n\delta_n}=e^{-2/u(n)}\to 1$. We may clearly
assume that $u(n)=O(n)$. Then,
$$\delta_n^{-1}\ln \left(2/\delta_n \right)n\ln n=n^2\ln n\ln (2nu(n))u(n)\approx (n\ln n)^2u(n),$$
and the result follows.\end{proof}

\bigskip

\noindent {\bf Acknowledgement.} We are grateful to the referee for extremely valuable comments and
suggestions on the presentation of the results of this article and for kindly communicating to us
the very short and elegant proof of Theorem~\ref{th:non-sharp} via \eqref{eq:ref-1}.
We acknowledge support by the Hellenic Foundation for Research and Innovation (H.F.R.I.) under the ``First Call for H.F.R.I.
Research Projects to support Faculty members and Researchers and the procurement of high-cost research equipment grant"
(Project Number: 1849).

\bigskip


\footnotesize

\bibliographystyle{amsplain}


\medskip

\thanks{\noindent {\bf Keywords:} log-concave probability measures, half-space depth, isotropic constant, random polytopes, convex bodies.

\smallskip

\thanks{\noindent {\bf 2010 MSC:} Primary 60D05; Secondary 62H05, 46B06, 52A40, 52A23.}

\bigskip

\bigskip

\noindent \textsc{Silouanos \ Brazitikos}: Department of
Mathematics, National and Kapodistrian University of Athens, Panepistimioupolis 157-84,
Athens, Greece.

\smallskip

\noindent \textit{E-mail:} \texttt{silouanb@math.uoa.gr}

\bigskip

\noindent \textsc{Apostolos \ Giannopoulos}: Department of
Mathematics, National and Kapodistrian University of Athens, Panepistimioupolis 157-84,
Athens, Greece.

\smallskip

\noindent \textit{E-mail:} \texttt{apgiannop@math.uoa.gr}

\bigskip

\noindent \textsc{Minas \ Pafis}: Department of
Mathematics, National and Kapodistrian University of Athens, Panepistimioupolis 157-84,
Athens, Greece.

\smallskip

\noindent \textit{E-mail:} \texttt{mipafis@math.uoa.gr}

\bigskip

\end{document}